\documentclass[10pt]{article}
\usepackage{a4wide}
\usepackage{latexsym,stmaryrd}
\usepackage{amsmath,amssymb,amsthm,bm}
\usepackage[latin1]{inputenc}
\usepackage{t1enc, lmodern}
\usepackage{graphicx}
\usepackage{color}
\usepackage[margin=1in]{geometry}

\usepackage{enumerate}
\usepackage{graphicx}
\usepackage{url}

\usepackage{makeidx}

\usepackage{wrapfig}
\usepackage{amsmath}
\usepackage{amsfonts}
\usepackage{mathrsfs}
\usepackage{amssymb}
\usepackage{latexsym}
\usepackage{amsbsy}
\usepackage{color}
\usepackage{bbm}

\usepackage{mathrsfs}

\usepackage{wasysym}

\providecommand{\otherindexspace}[1]{}

\usepackage{amsthm}

\newtheorem{theorem}{Theorem}[section]

\newtheorem{lemma}[theorem]{Lemma}
\newtheorem{proposition}[theorem]{Proposition}
\newtheorem{remark}[theorem]{Remark}

\newtheorem{definition}[theorem]{Definition}

\newtheorem{assumption}[theorem]{Assumption}

\numberwithin{equation}{section}

\def\cal#1{\mathcal{#1}}

\def \H{\mathbb {H}}
\def \N{\mathbb {N}}
\def \R{\mathbb {R}}
\def \E{\mathbb {E}}
\def \F{\mathbb {F}}
\def \P{\mathbb {P}}
\def \mS{\mathcal{S}}
\def \mA{\mathcal{A}}
\def \mF{\mathcal{F}}

\def \mP{\mathcal{P}}
\def \mG{\mathcal{G}}

\usepackage{fancyhdr}

\newcommand{\ov}{\overline}

\usepackage{graphics}
\usepackage{graphicx}

\DeclareGraphicsExtensions{.eps,.bmp,.jpg,.pdf,.mps,.png,.gif}

\makeatletter
\def\titre{\@title}
\makeatother

\begin{document}
\vspace{5cm}
\title{Reflected scheme for doubly reflected BSDEs with jumps and RCLL obstacles}
\author{Roxana DUMITRESCU\thanks{CEREMADE,
Universit\'e Paris 9 Dauphine, CREST  and  INRIA Paris-Rocquencourt, email: {\tt roxana@ceremade.dauphine.fr}. The research leading to these results has received funding from the R\'egion Ile-de-France. }
\and
C\'eline LABART\thanks{LAMA, 
Universit\'e de Savoie, 73376 Le Bourget du Lac, France and  INRIA Paris-Rocquencourt,
email: {\tt celine.labart@univ-savoie.fr}, corresponding author}}

\maketitle

\begin{abstract} We introduce a discrete time reflected scheme to solve
  doubly reflected Backward Stochastic Differential Equations with jumps (in short DRBSDEs), driven by a Brownian motion and an independent
  compensated Poisson process. As in \cite{DL14}, we approximate the Brownian motion and the Poisson process by two random
walks, but contrary to this paper, we discretize directly the DRBSDE,
without using a penalization step. This gives us a fully implementable scheme,
which only
depends on one parameter of approximation: the number of time steps $n$ (contrary
to the scheme proposed in \cite{DL14}, which also depends on the penalization parameter). We prove the convergence of the scheme, and give some numerical examples.

\end{abstract}

\vspace{10mm}

\noindent{\bf Key words~:} Double barrier reflected BSDEs, Backward
stochastic differential equations with jumps, numerical
scheme.

\vspace{10mm}

\noindent{\bf MSC 2010 classifications~:} 60H10, 60H35, 60J75, 34K28.



%
%
%
%
%
%
%
%
%
%
%
%
%

\section{Introduction}

\quad 

\quad Non-linear backward stochastic differential equations (BSDEs in short)
have been introduced by Pardoux and Peng in the Brownian framework in their
seminal paper \cite{PP90} and then extended to the case of jumps by Tang and
Li \cite{TL94}.  BSDEs appear as a useful mathematical tool in finance
(hedging problems) and in stochastic control. Moreover, these stochastic
equations provide a probabilistic representation for the solution of
semilinear partial differential equations.  BSDEs have been extended to the
reflected case by El Karoui et al in \cite{EKPPQ97}. In their setting, one of
the components of the solution is forced to stay above a given barrier which
is a continuous adapted stochastic process. The main motivation is the pricing
of American options especially in constrained markets. The generalization to the case of two
reflecting barriers has been carried out by Cvitanic and Karatzas in
\cite{CK96}. It is well known that doubly reflected BSDEs (DRBSDEs in the
following) are related to Dynkin games and to the pricing of
Israeli options (or Game options). The extension to the
case of reflected BSDEs with jumps and one reflecting barrier with only inaccessible
jumps has been established by Hamad\`ene and Ouknine \cite{HO03}. Later on,
Essaky in \cite{Ess08} and Hamad\`ene and Ouknine in \cite{HO13} have extended
these results to a right-continuous left limited (RCLL) obstacle with
predictable and inaccessible jumps.  Results concerning existence and
uniqueness of the solution for doubly reflected BSDEs with jumps can be found
in \cite{CM08},\cite{DQS14},
\cite{HH06}, \cite{HW09} and \cite{EHO05}.\\

\quad  Numerical schemes for DRBSDEs driven by the Brownian motion have been
proposed by Xu in \cite{Xu11} (see also \cite{MPX02} and \cite{PX11}) and, in the Markovian framework, by Chassagneux in
\cite{C09}. In this paper, we are interested in numerically solving DRBSDEs
driven by a Brownian motion and an independent Poisson process in the case of
RCLL obstacles with only totally inacessible jumps. More precisely, we consider equations of the following form:

\begin{align}\label{eqintro}
  \left \lbrace \begin{tabular}{l}
    \mbox{(i) $Y_t=\xi_T+\int_t^T g(s,Y_s,Z_s,U_s)ds+(A_T-A_t)-(K_T-K_t)-\int_t^T Z_s
      dW_s-\int_t^T U_s d\tilde{N}_s$},\\
    \mbox{(ii) $\forall t \in [0,T]$, $\xi_t \le Y_t \le \zeta_t$ a.s.,}\\
    \mbox{(iii) $\int_0^T (Y_{t} -\xi_{t}) dA_t=0$ a.s.  and $\int_0^T (\zeta_{t} - Y_{t})
      dK_t=0$ a.s.}
  \end{tabular} \right.
\end{align}  $\{W_t: 0 \leq t \leq T \}$ is a one dimensional
standard Brownian motion and $\{\Tilde{N}_t:=N_t-\lambda t, 0 \le t \le T\}$ is a compensated
Poisson process. Both processes are independent and they
are defined on the probability space $(\Omega,
\mathcal{F}_T,\mathbb{F}=\{\mathcal{F}_t\}_{0\le t \le T},\mathbb{P})$. The
processes $A$ and $K$ have the role to keep the solution between the two
obstacles $\xi$ and $\zeta$. Since we consider that the jumps of the obstacles
are totally inaccessible, $A$ and $K$ are continuous processes.\\

In the non-reflected case, some numerical methods have been provided: in
\cite{BE08}, the authors propose a scheme for Forward-Backward SDEs based on
the dynamic programming equation and in \cite{LMT07}, the authors propose a
fully implementable scheme based on a random binomial tree.  In the reflected
case, a fully implementable numerical scheme has been recently provided by
Dumitrescu and Labart in \cite{DL14}. Their method is based on the
approximation of the Brownian motion and the Poisson process by two random
walks and on the approximation of the reflected BSDE by a sequence of
penalized BSDEs.\\

The aim of this paper is to propose an alternative scheme to \cite{DL14} to
solve \eqref{eqintro}. The scheme proposed here takes the following form:
\begin{equation}\label{eqintro1}
\begin{cases}
\overline{y}_j^n=\E[\overline{y}_{j+1}^n|\mathcal{F}_j^n]+g(t_j,\E[\overline{y}_{j+1}^n|\mathcal{F}_j^n], \overline{z}_j^n, \overline{u}_j^n)\delta+\overline{a}_j^n-\overline{k}_j^n,\\
\overline{a}_j^n \geq 0,\; \overline{k}_j^n \geq 0,\; \overline{a}_j^n \overline{k}_j^n=0,\\
\xi_j^n \leq \overline{y}_j^n \leq \zeta_j^n,\; (\overline{y}_j^n-\xi_j^n)\overline{a}_j^n=(\overline{y}_j^n-\zeta_j^n)\overline{k}_j^n=0.\\
\end{cases}
\end{equation}

It generalizes the scheme proposed by \cite{Xu11} to the case of jumps.
Compared to the scheme proposed in \cite{DL14}, the scheme proposed here \textemdash called
\textit{reflected scheme} in the following \textemdash is based on the direct
discretization of \eqref{eqintro}. In particular, there is no penalization step. Then, this method only
depends on one parameter of approximation (the number of time steps $n$), contrary
to the scheme proposed in \cite{DL14} (which also depends on the penalization parameter). We
provide here an \textit{explicit reflected scheme} and an \textit{implicit
  reflected scheme} and we show the convergence of both schemes. We illustrate
numerically  the theoretical results and show  they coincide with the ones obtained by using
the penalized scheme presented in \cite{DL14}, for large values of the penalization parameter.

\quad The paper is organized as follows: in Section 2 we introduce notations
and assumptions. In Section 3, we precise the discrete time framework and
present the numerical schemes. In Section 4 we provide the convergence of the
schemes. Numerical examples are given in Section 5 .

\section{Notations and assumptions}

In this Section we introduce notations and assumptions. We recall the
result on existence and uniqueness of solution to \eqref{eqintro}. We also
introduce some assumptions on the obstacles $\xi$ and $\zeta$ specific to this
paper (Assumption \ref{hypo2}).\\

Let $(\Omega,  \F, \P)$ be a probability space, and ${\mP}$ be  the predictable $\sigma$-algebra
on $[0,T]  \times \Omega$.
Let  $W$ be a one-dimensional Brownian motion and  $N$ be a Poisson process with
intensity $\lambda >0$. Let  $\F = \{\mathcal{F}_t , 0\leq t \leq T \}$ 
be  the natural filtration associated with $W$ and $N$. \\

For each $T>0$, we use the following notations:
\begin{itemize}
\item
$L^2(\mathcal{F}_T)$  is the set of $ \mathcal{F}_T
$-measurable and square integrable random variables.

\item    $\H^{2}$ is the set of
real-valued predictable processes $\phi$ such that $\| \phi\|^2_{\H^{2}} := \E \left[\int_0 ^T \phi_t ^2 dt \right] < \infty.$

\item $\mathcal {B}(\R^2)$ is the Borelian
$\sigma$-algebra on $\R^2$.

\item   ${\cal S}^{2}$ is the set of real-valued RCLL adapted
processes $\phi$ such that $\| \phi\|^2_{\mathcal {S}^2} := \E(\sup_{0\leq t \leq T} |\phi_t |^2) <  \infty.$

\item $\mathcal {A}^2$ is the set of real-valued non decreasing RCLL predictable
processes $A$ with $A_0 = 0$ and $\E(A^2_T) < \infty$. 
\end {itemize}

\begin{definition}[Driver, Lipschitz driver]\label{defd}
A function $g$ is said to be a {\em driver} if
\begin{itemize}
\item
$g:  \Omega \times [0,T] \times \R^3 \rightarrow \R $\\
$(\omega, t,y, z, u) \mapsto  g(\omega, t,y,z,u) $
 is $ {\cal P} \otimes {\cal B}(\R^3)$-measurable,
\item $\|g(.,0,0,0)\|_{\infty}< \infty$.
\end{itemize}
A driver $g$ is called a {\em Lipschitz driver} if moreover there exists a
constant $ C_g \geq 0$ and a bounded, non-decreasing continuous function
$\Lambda$ with $\Lambda(0)=0$ such that $d \P \otimes dt$-a.s.\,,
for each $(s_1,y_1, z_1, u_1)$, $(s_2,y_2, z_2, u_2)$,
$$|g(\omega, s_1, y_1, z_1, u_1) - g(\omega, s_2, y_2, z_2, u_2)| \leq \Lambda(|s_2-s_1|)+C_g (|y_1 - y_2| + |z_1 - z_2| + |u_1 - u_2 |).$$
\end{definition}

\begin{definition}[Mokobodzki's condition] Let $\xi$, $\zeta$ be in
  $\mS^2$. There exist two nonnegative RCLL supermartingales $H$ and $H'$ in
  $\mS^2$ such that
  \begin{align*}
\forall t \in [0,T],\;\; \xi_t  \le H_t-H'_t \le \zeta_t
\mbox{ a.s.}
  \end{align*}
\end{definition}

The following Theorem states existence and uniqueness of solutions to \eqref{eqintro} (see for e.g. \cite[Proposition 5.1]{CM08}).
\begin{theorem}\label{thm1}
  Suppose $\xi$ and $\zeta$ are RCLL adapted processes in $\mS^2$ such that
  for all $t\in [0,T]$, $\xi_t\le \zeta_t$, Mokobodzki's condition
  holds and $g$ is a Lipschitz driver. Then, DRBSDE \eqref{eqintro} admits a unique solution $(Y,Z,U,\alpha)$ in
  $\mS^2\times \H^2 \times \H^2 \times \mS^2$, where $\alpha:=A-K$, $A$ and
  $K$ in $\mA^2$.
\end{theorem}

Let us now introduce an additional assumption on $g$, which ensures the
comparison theorem for BSDEs with jumps (see \cite[Theorem 4.2]{QS13}). The
comparison theorem plays a key role in the proof of the convergence of the
penalized scheme (see \cite{DL14}), which is useful to prove the convergence
of the reflected scheme (see Section \ref{sect:conv_res}).
  
\begin{assumption}\label{hypo1} A Lipschitz driver $g$ is said to satisfy
  Assumption \ref{hypo1} if the following holds : $d\P \otimes dt $ a.s. for
  each $(y,z,u_1,u_2) \in \R^4$, we have
  \begin{align*}
    g(t,y,z,u_1)-g(t,y,z,u_2) \ge \theta(u_1-u_2), \mbox{ with }-1 \le \theta \le \theta_0.
  \end{align*}
\end{assumption}

We also assume the following hypothesis on the barriers. 

\begin{assumption}\label{hypo2} $\xi$ and $\zeta$ are Itô processes of the
  following form
\begin{equation}\label{barrierxi}
\xi_t=\xi_0+\int_0^tb^{\xi}_sds+\int_0^t\sigma^{\xi}_{s}dW_s+\int_0^t\beta^{\xi}_{s^-}d\Tilde{N}_s
\end{equation}
\begin{equation}\label{barrierzeta}
\zeta_t=\zeta_0+\int_0^t
b^{\zeta}_sds+\int_0^t\sigma^{\zeta}_{s}dW_s+\int_0^t\beta^{\zeta}_{s^-} d\Tilde{N}_s
\end{equation}
where $b^{\xi}$, $b^{\zeta}$, $\sigma^{\xi}$, $\sigma^{\zeta}$, $\beta^{\xi}$
and $\beta^{\zeta}$ are adapted RCLL processes
such that there exists $r>2$ and a constant $C_{\xi,\zeta}$ such that
$\E(\sup_{s \le T} |b^{\xi}_s|^r) + \E(\sup_{s \le T} |b^{\zeta}_s|^r) +
\E(\sup_{s \le T} |\sigma^{\xi}_s|^r) + \E(\sup_{s \le T}
|\sigma^{\zeta}_s|^r) + \E(\sup_{s \le T} |\beta^{\xi}_s|^r) + \E(\sup_{s \le
  T} |\beta^{\zeta}_s|^r) \le C_{\xi,\zeta}$. We also assume $\xi_T=\zeta_T$
a.s., $\xi_t \le \zeta_t$ for all $t \in [0,T]$.
\end{assumption}

\section{Discrete time framework and numerical scheme}
 
\subsection{Discrete time framework}

For the numerical part of the paper, we adopt the framework of
\cite{LMT07} and \cite{DL14}, presented below.

\subsubsection{Random walk approximation of $(W,\tilde{N})$}
For $n \in \mathbb{N}$, we introduce $\delta:=\frac{T}{n}$ and the regular grid
$(t_j)_{j=0,...,n}$ with step size $\delta$ (i.e. $t_j:=j \delta$) to
discretize $[0,T]$. In order to approximate $W$, we introduce the following
random walk
\begin{equation}
\begin{cases}
W_0^n=0,\\
W_t^n=\sqrt{\delta} \sum_{i=1}^{[t/\delta]} e_i^n, 
\end{cases}
\end{equation}
where $e_1^n, e_2^n,...,e_n^n$ are independent identically distributed random
variables with the following symmetric Bernoulli law:
$$\P(e_1^n=1)=\P(e_1^n=-1)=\frac{1}{2}.$$
To approximate $\tilde{N}$, we introduce a second random walk
\begin{equation}
\begin{cases}
\Tilde{N}_0^n=0,\\
\Tilde{N}_t^n=\sum_{i=1}^{[t/\delta]}\eta_i^n, 
\end{cases}
\end{equation}
where $\eta_1^n, \eta_2^n,...,\eta_n^n$ are independent and identically distributed random variables with law 
$$\P(\eta_1^n=\kappa_n-1)=1-\P(\eta_1^n=\kappa_n)=\kappa_n, $$
where $\kappa_n=e^{-\lambda \delta}.$
We assume that both sequences $e_1^n,...,e_n^n$ and $\eta_1^n,
\eta_2^n,...,\eta_n^n$ are defined on the original probability space $(\Omega,
\mathcal{F}, \mathbb{P}).$ The (discrete) filtration in the probability space
is $\mathbb{F}^n:=\{\mathcal{F}_j^n: j=0,...,n \}$ with
$\mathcal{F}_0^n=\{\Omega,\emptyset \}$ and
$\mathcal{F}_j^n=\sigma\{e_1^n,...,e_j^n, \eta_1^n,...,\eta_j^n\}$ for
$j=1,...,n.$

The following result states the convergence of $(W^n,\tilde{N}^n)$ in the
$J_1$-Skorokhod topology. We refer to \cite[Section 3]{LMT07} for more results
on the convergence in probability of $\mF^n$-martingales .

\begin{lemma}(\cite[Lemma3, (III)]{LMT07}
  The couple $(W^n,\tilde{N}^n)$ converges in probability to $(W,\tilde{N})$
  for the $J_1$-Skorokhod topology.
\end{lemma}
We recall that the process $\xi^n$ converges in probability to $\xi$ in the
$J_1$-Skorokhod topology if there exists a family
  $(\psi^n)_{n\in \N}$ of one-to-one random time changes from $[0,T]$ to
  $[0,T]$ such that $\sup_{t \in [0,T]}|\psi^n(t)-t| \xrightarrow[n\rightarrow
  \infty]{}0$ almost surely and $\sup_{t \in [0,T]}
|{\xi}^n_{\psi^n(t)}-\xi_t|\xrightarrow[n \rightarrow \infty]{}0$ in probability.

\subsubsection{Martingale representation}

Let $y_{j+1}$ denote a $\mF^n_{j+1}$-measurable random variable. As pointed out in
\cite{LMT07}, we need a set of three strongly orthogonal martingales to
represent the martingale difference $m_{j+1}:=y_{j+1}-\E(y_{j+1}|\mF^n_j)$. We
introduce a third martingale increment sequence
$\{\mu^n_j=e^n_j\eta^n_j,j=0,\cdots,n\}$. In this context there exists a
unique triplet $(z_j,u_j,v_j)$ of $\mF^n_j$-random variables such that
\begin{align*}
  m_{j+1}:=y_{j+1}-\E(y_{j+1}|\mF^n_j)=\sqrt{\delta} z_j e^n_{j+1} + u_j
  \eta^n_{j+1} +v_j \mu^n_{j+1},
\end{align*}
and
\begin{align}\label{eq27}
  \left\{ \begin{array}{l}
      z_j=\frac{1}{\sqrt{\delta}}\E(y_{j+1}e^n_{j+1}| \mF^n_j),\\
      u_j=\frac{\E(y_{j+1}\eta^n_{j+1}|\mF^n_j)}{\E((\eta^n_{j+1})^2|\mF^n_j)}=\frac{1}{\kappa_n(1-\kappa_n)}\E(y_{j+1}\eta^n_{j+1}|
      \mF^n_j),\\
      v_j=\frac{\E(y_{j+1}\mu^n_{j+1}|\mF^n_j)}{\E((\mu^n_{j+1})^2|\mF^n_j)}=\frac{1}{\kappa_n(1-\kappa_n)}\E(y_{j+1}\mu^n_{j+1}|
      \mF^n_j).\\
    \end{array}\right.
\end{align}

The computation of conditional expectations is done in the following way:
\begin{remark}(Computing the conditional expectations)\label{rem:cond_exp} Let $\Phi$ denote a
  function from $\R^{2j+2}$ to $\R$. We use the following formula
  \begin{align*}
    \E(\Phi(e^n_1,\cdots,e^n_{j+1},\eta^n_1,\cdots,\eta^n_{j+1})|\mF^n_j)=&\frac{\kappa_n}{2}\Phi(e^n_1,\cdots,e^n_{j},1,\eta^n_1,\cdots,\eta^n_{j},\kappa_n-1)\notag\\
    &+\frac{\kappa_n}{2}\Phi(e^n_1,\cdots,e^n_{j},-1,\eta^n_1,\cdots,\eta^n_{j},\kappa_n-1)\notag\\
    &+\frac{1-\kappa_n}{2}\Phi(e^n_1,\cdots,e^n_{j},1,\eta^n_1,\cdots,\eta^n_{j},\kappa_n)\notag\\
    &+\frac{1-\kappa_n}{2}\Phi(e^n_1,\cdots,e^n_{j},-1,\eta^n_1,\cdots,\eta^n_{j},\kappa_n).
  \end{align*}
\end{remark}

\subsection{Reflected schemes}

The barriers $\xi$ and $\zeta$ given in Assumption \ref{hypo2} are
approximated in the following way: for all $k \in \{1,\cdots,n\}$

\begin{equation}\label{xixi}
\xi^n_k=\xi_0+\sum_{i=0}^{k-1}b^{\xi}_{t_i}\delta+\sum_{i=0}^{k-1}\sigma^{\xi}_{t_i}\sqrt{\delta}e_{i+1}^n+\sum_{i=0}^{k-1}\beta^{\xi}_{t_i}\eta_{i+1}^n,
\end{equation}
\begin{equation}
\zeta^n_k=\zeta_0+\sum_{i=0}^{k-1}b^{\zeta}_{t_i}\delta+\sum_{i=0}^{k-1}\sigma^{\zeta}_{t_i}\sqrt{\delta}e_{i+1}^n+\sum_{i=0}^{k-1}\beta^{\zeta}_{t_i}\eta_{i+1}^n.
\end{equation}

\begin{lemma}\label{lem6}
  Under Assumption \ref{hypo2}, there exists a constant
  $C_{\xi,\zeta,T,\lambda}$ depending on $C_{\xi,\zeta}$, $T$ and $\lambda$ such
  that 
  \begin{align*}
    & \mbox{(i)} \sup_n \max_j
    \E(|\xi^n_j|^r)+\sup_n \max_j \E(|\zeta^n_j|^r)+\sup_{t \le T}
    \E(|\xi_t|^r)  +\sup_{t \le T}
    \E(|\zeta_t|^r) \le C_{\xi,\zeta,T,\lambda}
    \\
    & \mbox{(ii) $\xi^n$ (resp. $\zeta^n$) converges in probability to $\xi$
      (resp. $\zeta$) in $J_1$-Skorokhod topology}. 
  \end{align*}
\end{lemma}

\begin{proof} $(i)$ ensues from Burkhölder-Davis-Gundy and Rosenthal
  inequalities, and $(ii)$ ensues from \cite[Theorem 6.22 and Corollary 6.29]{JS_03}.
\end{proof}

In the following Section we introduce the implicit reflected scheme, which is
an intermediate scheme useful to prove the convergence of the reflected scheme
\eqref{eqintro1}.

\subsubsection{Implicit reflected scheme}

After the discretization of the time interval, our discrete reflected BSDEs with two RCLL barriers on small interval $[t_j, t_{j+1}[,$ for $0 \leq j \leq n-1$ is 
\begin{equation}\label{scheme2}
\begin{cases}
y_j^n=y_{j+1}^n+g(t_j,y_j^n,z_j^n,u_j^n)\delta+a_j^n-k_j^n-z_j^n \sqrt \delta \varepsilon_{j+1}^n-u_j^n \eta_{j+1}^n-v_j^n \mu_{j+1}^n,\\
a_j^n \geq 0,\; k_j^n \geq 0,\; a_j^n k_j^n=0,\; 
\xi_j^n \leq y_j^n \leq \zeta_j^n,\; (y_j^n-\xi_j^n)a_j^n=(y_j^n-\zeta_j^n)k_j^n=0.
\end{cases}
\end{equation}
with terminal condition $y_n^n=\xi^n_n.$ By taking the conditional expectation
in \eqref{scheme2} w.r.t. $\mathcal{F}^n_j$, we get

\begin{equation*}
  (\mathcal{S}_1)\begin{cases}
    y^n_n=\xi^n_n,\\
y_j^n=\E[y_{j+1}^n|\mathcal{F}_j^n]+g(t_j,y_j^n,z_j^n,u_j^n)\delta+a_j^n-k_j^n,\\
a_j^n \geq 0,\; k_j^n \geq 0,\; a_j^n k_j^n=0, \\
\xi_j^n \leq y_j^n \leq \zeta_j^n,\; (y_j^n-\xi_j^n)a_j^n=(y_j^n-\zeta_j^n)k_j^n=0.
\end{cases}
\end{equation*}

\begin{lemma}For $\delta$ small enough, $(\mathcal{S}_1)$ is equivalent to
\begin{equation*}
  (\mathcal{S}_2)\begin{cases}
    y^n_n=\xi^n_n,\\
    y_j^n=\Psi^{-1}(\E[y_{j+1}^n|\mathcal{F}_j^n]+a_j^n-k_j^n),\\
    a_j^n=(\E[y_{j+1}^n|\mathcal{F}_j^n]+g(t_j,\xi_j^n,z_j^n,u_j^n)\delta-\xi_j^n)^{-},\\
    k_j^n=(\E[y_{j+1}^n|\mathcal{F}_j^n]+g(t_j,\zeta_j^n,z_j^n,u_j^n)\delta-\zeta_j^n)^{+},\\
\end{cases}
\end{equation*}
where $\Psi(y):=y-g(t_j,y,z_j^n,u^n_j)\delta.$
\end{lemma}

\begin{proof} For $\delta$ small enough, $\Psi$ is invertible because
  the Lipschitz property of $g$ leads to $(\Psi(y)-\Psi(y'))(y-y') \ge
  (1-\delta C_g)(y-y')^2 > 0$ for any $y\neq y'$. \\
  We first prove that $(\mathcal{S}_1)$ implies $(\mathcal{S}_2)$. Let us
  firstly assume that $\forall j \le n-1$,
  $\xi^n_j <\zeta^n_j$. On the set
  $\{y^n_j=\xi^n_j\}$ we have $k^n_j=0$, then
  $a^n_j=\Psi(\xi^n_j)-\E[y_{j+1}^n|\mathcal{F}_j^n]=(\E[y_{j+1}^n|\mathcal{F}_j^n]-\Psi(\xi^n_j))^-$
  (since
  $\E[y_{j+1}^n|\mathcal{F}_j^n]-\Psi(\xi^n_j)=\Psi(y^n_j)-\Psi(\xi^n_j)-a^n_j
  \le 0)$ and on
  $\{y^n_j>\xi^n_j\}$ we have $a^n_j =0$,
  $(\E[y_{j+1}^n|\mathcal{F}_j^n]-\Psi(\xi^n_j)=\Psi(y^n_j)-\Psi(\xi^n_j)+k^n_j>0$
  (thanks to the monotonicity of $\Psi$)).
  Then, $a^n_j=(\E[y_{j+1}^n|\mathcal{F}_j^n]-\Psi(\xi^n_j))^-$.
  The same
  type of proof leads to the fourth line of $(\mathcal{S}_2)$. If there exists
  $j\le n-1$ such that
  $\xi^n_j=\zeta^n_j$, we get $\xi^n_j=\zeta^n_j= y^n_j$. Then, we have
  $a^n_j=0$ or $k^n_j=0$. If both are null, we get
  $\Psi(y^n_j)=\E[y_{j+1}^n|\mathcal{F}_j^n]=\Psi(\xi^n_j)=\Psi(\zeta^n_j)$. This
  coincides with the definitions of $a^n_j$ and $k^n_j$ given in
  $(\mathcal{S}_2)$. If $a^n_j>0$, $k^n_j=0$ and we get
  $a^n_j=\Psi(y^n_j)-\E[y_{j+1}^n|\mathcal{F}_j^n]=\Psi(\xi^n_j)-\E[y_{j+1}^n|\mathcal{F}_j^n]$,
  then $a^n_j=(\E[y_{j+1}^n|\mathcal{F}_j^n]-\Psi(\xi^n_j))^-$. Conversely,
  assume $(\mathcal{S}_2)$, let us prove $a^n_j
  k^n_j=0$, $(y^n_j-\xi^n_j)a^n_j=(y^n_j-\zeta^n_j)k^n_j=0$ and $\xi^n_j \le y^n_j \le \zeta^n_j$. If $a^n_j>0$, we get $\Psi(\zeta^n_j)\ge
  \Psi(\xi^n_j)>\E[y_{j+1}^n|\mathcal{F}_j^n]$, then $k^n_j=0$. Let us prove
  that $(y^n_j-\xi^n_j)a^n_j=0$. If $a^n_j>0$, $\Psi(y^n_j)=\E[y_{j+1}^n|\mathcal{F}_j^n]+a^n_j=\Psi(\xi^n_j)$. Since
  $\Psi$ is a one to one map, we get $y^n_j=\xi^n_j$. The same argument
  holds to prove $(y^n_j-\zeta^n_j)k^n_j=0$. Let us prove that $\xi^n_j \le
  y^n_j $. To do so, assume that $y^n_j < \xi^n_j$. In this case $a^n_j=k^n_j=0$,  which gives $\Psi( \xi^n_j) \le \E[y_{j+1}^n|\mathcal{F}_j^n]$, by
  definition of $a^n_j$. Then $\Psi(y^n_j)=\E[y_{j+1}^n|\mathcal{F}_j^n] \ge
  \Psi(\xi^n_j)$. $\Psi$ being a non decreasing function, this leads to absurdity.
\end{proof}

We also introduce the continuous time version $({Y}_t^{n},
{Z}_t^{n}, {U}_t^{n}, {A}_t^{n},
{K}_t^{n})_{ 0 \leq t \leq T}$ of $(y^n_j,z^n_j,u^n_j,a^n_j,k^n_j)_{j \le n}$:
\begin{align}\label{eq34}
  {Y}_t^{n}:={y}_{[t/\delta]}^{n}, {Z}_t^{n}:={z}_{[t/\delta]}^{n},
  {U}_t^{n}:={u}_{[t/\delta]}^{n}, 
  {A}_t^{n}:=\sum_{i=0}^{[t/\delta]}{a}_i^{n}, {K}_t^{n}:=
  \sum_{i=0}^{[t/\delta]}{k}_i^{n}.
\end{align}
In the following $\Theta^n:=(Y^n,Z^n,U^n,A^n-K^n)$.

\subsubsection{Explicit reflected scheme}

The explicit reflected scheme is introduced by replacing $y_j^n$ by
$\E[\overline{y}_{j+1}^n| \mathcal{F}_j^n]$ in $g$. We obtain
\begin{equation}\label{scheme2_exp}
\begin{cases}
\ov{y}_j^n=\ov{y}_{j+1}^n+g(t_j,\E[\overline{y}_{j+1}^n|\mathcal{F}_j^n],
\overline{z}_j^n,\overline{u}_j^n)\delta+\overline{a}_j^n-\overline{k}_j^n-\ov{z}_j^n
\sqrt \delta \varepsilon_{j+1}^n-\ov{u}_j^n \eta_{j+1}^n-\ov{v}_j^n
\mu_{j+1}^n,\\
\overline{a}_j^n \geq 0, \; \overline{k}_j^n \geq 0,\; \overline{a}_j^n \overline{k}_j^n=0,\\
\xi_j^n \leq \overline{y}_j^n \leq \zeta_j^n,\; (\overline{y}_j^n-\xi_j^n)\overline{a}_j^n=(\overline{y}_j^n-\zeta_j^n)\overline{k}_j^n=0.\\
\end{cases}
\end{equation}
with terminal condition $\ov{y}_n^n=\xi^n_n.$
By taking the conditional expectation in $\eqref{scheme2_exp}$ with respect to
$\mathcal{F}_j^n$, we derive that:

\begin{equation*}
  (\overline{\mathcal{S}}_1)\begin{cases}
    \ov{y}^n_n=\xi^n_n,\\
    \overline{y}_j^n=\E[\overline{y}_{j+1}^n|\mathcal{F}_j^n]+g(t_j,\E[\overline{y}_{j+1}^n|\mathcal{F}_j^n], \overline{z}_j^n, \overline{u}_j^n)\delta+\overline{a}_j^n-\overline{k}_j^n\,\
    \overline{a}_j^n \geq 0,\; \overline{k}_j^n \geq 0,\; \overline{a}_j^n \overline{k}_j^n=0,\\
    \xi_j^n \leq \overline{y}_j^n \leq \zeta_j^n,\; (\overline{y}_j^n-\xi_j^n)\overline{a}_j^n=(\overline{y}_j^n-\zeta_j^n)\overline{k}_j^n=0.\\
  \end{cases}
\end{equation*}
As for the implicit reflected scheme, we get that $(\overline{\mathcal{S}}_1)$
is equivalent to $(\overline{\mathcal{S}}_2)$
\begin{align*}
 (\overline{\mathcal{S}}_2) \begin{cases}
    \ov{y}^n_n=\xi^n_n,\\
    \overline{y}_j^n=\E[\overline{y}_{j+1}^n|\mathcal{F}_j^n]+g(t_j,\E[\overline{y}_{j+1}^n|\mathcal{F}_j^n], \overline{z}_j^n, \overline{u}_j^n)\delta+\overline{a}_j^n-\overline{k}_j^n,\\
    \overline{a}_j^n=(\E[\overline{y}_{j+1}^n|\mathcal{F}_j^n]+g(t_j,\E[\overline{y}_{j+1}^n|\mathcal{F}_j^n], \overline{z}_j^n, \overline{u}_j^n)\delta-\xi_j^n)^{-},\\
    \overline{k}_j^n=(\E[\overline{y}_{j+1}^n|\mathcal{F}_j^n]+g(t_j,\E[\overline{y}_{j+1}^n|\mathcal{F}_j^n], \overline{z}_j^n, \overline{u}_j^n)\delta-\zeta_j^n)^{+}.
  \end{cases}
\end{align*}

We also introduce the continuous time version $(\overline{Y}_t^{n},
\overline{Z}_t^{n}, \overline{U}_t^{n}, \overline{A}_t^{n},
\overline{K}_t^{n})_{ 0 \leq t \leq T}$ of
$(\ov{y}^n_j,\ov{z}^n_j,\ov{u}^n_j,\ov{a}^n_j,\ov{k}^n_j)_{j\le n}$:
\begin{align}\label{eq33}
  \ov{Y}_t^{n}:=\ov{y}_{[t/\delta]}^{n}, \ov{Z}_t^{n}:=\ov{z}_{[t/\delta]}^{n},
  \ov{U}_t^{n}:=\ov{u}_{[t/\delta]}^{n}, 
  \ov{A}_t^{n}:=\sum_{i=0}^{[t/\delta]}\ov{a}_i^{n}, \ov{K}_t^{n}:=
  \sum_{i=0}^{[t/\delta]}\ov{k}_i^{n}.
\end{align}

In the following
$\ov{\Theta}^n:=(\ov{Y}^n,\ov{Z}^n,\ov{U}^n,\ov{A}^n-\ov{K}^n)$ and $\ov{\alpha}^n:=\ov{A}^n-\ov{K}^n$.

\subsection{Implicit penalization scheme}\label{sect:imp_sch}

In this Section we recall the {\it implicit penalization scheme} introduced in
\cite{DL14}. The penalization is represented by the parameter $p$. As the
implicit reflected scheme, this scheme will be useful to prove the convergence
of the explicit reflected scheme. For all $j$ in $\{0,\cdots,n-1\}$ we have

\begin{equation}\label{discrete}
\begin{cases}
y_{j}^{p,n}=y_{j+1}^{p,n}+g(t_j, y_j^{p,n}, z_j^{p,n}, u_j^{p,n}) \delta+a_j^{p,n}-k_j^{p,n}-(z_j^{p,n}\sqrt{\delta}e_{j+1}^n+u_j^{p,n}\eta_{j+1}^n+v_j^{p,n}\mu_{j+1}^n),\\
a_j^{p,n}=p \delta(y_j^{p,n}-\xi_j^n)^{-} \text{,   }k_j^{p,n}=p \delta(\zeta_j^n-y_j^{p,n})^{-},\\
y_n^{p,n}:=\xi^n_n.
\end{cases}
\end{equation}

Following \eqref{eq27}, the triplet $(z_j^{p,n}, u_j^{p,n}, v_j^{p,n})$ can be
computed as follows 

\begin{align*}
  \left\{ \begin{array}{l}
      z_j^{p,n}=\frac{1}{\sqrt{\delta}}\E(y^{p,n}_{j+1}e^n_{j+1}| \mF^n_j),\\
      u_j^{p,n}=\frac{1}{\kappa_n(1-\kappa_n)}\E(y^{p,n}_{j+1}\eta^n_{j+1}|
      \mF^n_j),\\
      v^{p,n}_j=\frac{1}{\kappa_n(1-\kappa_n)}\E(y^{p,n}_{j+1}\mu^n_{j+1}|
      \mF^n_j).\\
    \end{array}\right.
\end{align*}
 Taking the conditional expectation w.r.t. $\mF^n_j$ in
\eqref{discrete}, we get
\begin{align*}
  \begin{cases}
    y^{p,n}_j=(\Psi^{p,n})^{-1}(\E(y^{p,n}_{j+1}|\mF^n_j)),\\
    a^{p,n}_j=p\delta (y^{p,n}_j-\xi^n_j)^-; \,
    k^{p,n}_j=p\delta(\zeta^n_j-y^{p,n}_j)^-,\\
    z_j^{p,n}=\frac{1}{\sqrt{\delta}}\E(y^{p,n}_{j+1}e^n_{j+1}| \mF^n_j),\\
      u_j^{p,n}=\frac{1}{\kappa_n(1-\kappa_n)}\E(y^{p,n}_{j+1}\eta^n_{j+1}|
      \mF^n_j),\\
      \end{cases}
\end{align*}
where $\Psi^{p,n}(y)=y-g(j\delta, y, z_j^{p,n},
u_j^{p,n})\delta-p\delta(y-\xi_j^n)^{-}+p\delta (\zeta_j^n-y)^{-}$.

We also introduce the continuous time version $(Y_t^{p,n}, Z_t^{p,n}, U_t^{p,n}, A_t^{p,n}, K_t^{p,n})_{ 0 \leq t \leq T}$ of the solution
of the discrete equation \eqref{discrete}:
\begin{align}\label{eq32}
  Y_t^{p,n}:=y_{[t/\delta]}^{p,n}, Z_t^{p,n}:=z_{[t/\delta]}^{p,n},
  U_t^{p,n}:=u_{[t/\delta]}^{p,n}, 
  A_t^{p,n}:=\sum_{i=0}^{[t/\delta]}a_i^{p,n}, K_t^{p,n}:=
  \sum_{i=0}^{[t/\delta]}k_i^{p,n},
\end{align}
and $\alpha^{p,n}:=A^{p,n}-K^{p,n}$.
 The following result ensues from \cite[Theorem 4.1 and Proposition 4.2]{DL14}.
  \begin{theorem}\label{main_thm}
Assume that Assumption \ref{hypo2} holds and $g$ is a Lipschitz driver satisfying
Assumption \ref{hypo1}. The sequence
$({Y}^{p,n}, {Z}^{p,n}, {U}^{p,n})$ defined by \eqref{eq32}
converges to $(Y,Z,U)$, the solution of the DRBSDE \eqref{eqintro}, in the
following sense: $\forall r \in [1,2[$
\begin{align}\label{conv1}
\lim_{p \rightarrow \infty} \lim_{n \rightarrow
  \infty}\left( \mathbb{E}\left[\int_0^T|{Y}_{s}^{p,n}-Y_s|^2ds\right] +
\mathbb{E}\left[\int_0^T|{Z}_{s}^{p,n}-Z_s|^rds\right]+ \mathbb{E}\left[\int_0^T|{U}_{s}^{p,n}-U_s|^rds\right]\right) = 0.
\end{align}
Moreover, ${Z}^{p,n}$ (resp. ${U}^{p,n}$) weakly converges in $\H^2$ to $Z$
(resp. to $U$) and for $0 \leq t \leq T$, ${\alpha}_{\psi^n(t)}^{p,n}$ converges weakly to  $\alpha_t$ in
$L^2(\mathcal{F}_T)$ as $n \rightarrow \infty$ and $p \rightarrow
\infty$, where $(\psi^n)_{n \in \N}$ is a one-to-one random map from $[0,T]$
to $[0,T]$ such that $\sup_{t \in [0,T]} |\psi^n(t)-t|\xrightarrow[n
  \rightarrow \infty]{} 0$ a.s.. 
\end{theorem}

\section{Convergence result}\label{sect:conv_res}

We prove in this Section that $\overline{\Theta}^{n}$ converges to $\Theta:=({Y}_t, {Z}_t, {U}_t, {A}_t-{K}_t)_{0\le
t \le T}$, the solution to the DRBSDE \eqref{eqintro}. The main result is
stated in the following Theorem.

\begin{theorem}\label{Schema2refl}
  Suppose that Assumption \ref{hypo2} holds and $g$ is a Lipschitz driver
  satisfying Assumption \ref{hypo1}. Then we have
  \begin{align*}
    \lim_{n \rightarrow \infty} \E\left[\int_0^T
      |\overline{Y}_t^n-Y_t|^2dt+\int_0^T |\overline{Z}_t^n-Z_t|^2dt+\int_0^T |\overline{U}_t^n-U_t|^2dt \right] =0.
  \end{align*}
  Moreover, $\overline{\alpha}_{\psi^n(t)}^n$ converges weakly to $\alpha_t$ in
  $L^2(\mathcal{F}_T)$.\\
\end{theorem}

\begin{proof}
To prove this result, we split the error in three terms. The first one is the error
$\overline{\Theta}^n-\Theta^n$, the second one is $\Theta^n- \Theta^{p,n}$, where
$\Theta^{p,n}:=({Y}^{p,n},{Z}^{p,n},{U}^{p,n},
{A}^{p,n}-{K}^{p,n})$ represents the solution given by the implicit
penalization scheme (see \eqref{eq32}), and the third error term is
$\Theta^{p,n}-\Theta$, whose convergence has already been proved in
\cite{DL14}. The result on the convergence of $\Theta^{p,n}$ to $\Theta$ is recalled in Theorem \ref{main_thm}.\\

We have the following inequality for the error on $Y$ (the same inequality
holds for the errors on $Z$ and $U$)
\begin{align*}
\E[\int_0^T|\overline{Y}_{t}^n-Y_{t}|^2 dt] \leq 3\E[\int_0^T
|\overline{Y}_{t}^n-Y_{t}^n|^2 dt]+3\E[\int_0^T
|Y_{t}^n-Y_{t}^{p,n}|^2 dt]+3[\int_0^T |Y_{t}^{p,n}-Y_{t}|^2 dt].
\end{align*}

  For the increasing processes, we have:
\begin{align}\label{Aconv}
\E[|\ov{\alpha}_{\psi^n(t)}^{n}-\alpha_{t}|^2] \leq
3\left(E[|\ov{\alpha}_{\psi^n(t)}^{n}-\alpha_{\psi^n(t)}^{n}|^2]+\E[|\alpha_{\psi^n(t)}^{n}-\alpha^{p,n}_t|^2]+\E[|\alpha_t^{p,n}-\alpha_t|^2]\right).
\end{align}

Then, combining Propositions \ref{prop6}, \ref{prop7} and Theorem
\ref{main_thm} yields the result.
\end{proof}

\begin{definition}[Definition of $c$ and $N_0$]\label{def1}
  In this Section and in the Appendix, $c$ denotes a generic constant depending on $C_g$,
  $\|g(\cdot,0,0,0)\|_{\infty}$ and $C_{\xi,\zeta,\lambda,T}$. $N_0$ is
  defined by $N_0:=
  4T(1+C_g+C_g^2+C_g^2\frac{e^{2 \lambda T}}{\lambda})$.
\end{definition}

  The rest of the
  Section is organized as follows: Section \ref{sect:imp_sch} recalls
  the implicit penalization scheme introduced in \cite{DL14}
  and the convergence of $\Theta^{p,n}-\Theta$, we give some intermediate
  results in Section \ref{sect:int_res} and we prove the convergence of 
  $\ov{\Theta}^n-\Theta^n$ (see Proposition \ref{prop6}) and the convergence
  of 
  ${\Theta}^n-\Theta^{p,n}$ (see Proposition \ref{prop7}) in Section \ref{sect:proof_prop}.

\subsection{Intermediate results}\label{sect:int_res}

In this Section we state two intermediate results useful for Section \ref{sect:proof_prop}.
\begin{lemma}\label{Reflscheme} Under Assumption \ref{hypo2} we have
\begin{align*}
 \sup_j \E[ |y_j^n|^2]+\E\left[ \delta \sum_{j=0}^{n-1}|z_j^n|^2
   +\kappa_n(1-\kappa_n) \sum_{j=0}^{n-1}|u_j^n|^2  + \frac{1}{\delta}
   \sum_{j=0}^{n-1} |a_j^n|^2+ \frac{1}{\delta} \sum_{j=0}^{n-1} |k_j^n |^2\right] \leq c.
\end{align*}
\end{lemma}


\begin{proof}
  
  Since $\xi^n_j \le y^n_j \le \zeta^n_j$, Assumption \ref{hypo2} gives
  $\sup_j \E(|y^n_j|^2) \le c$. Let us deal with $z^n_j$ and $u^n_j$.
To do this, we apply Lemma \ref{lem12} with $i_0=i$ and $i_1={i+1}$ to the process $y^n$
and we sum the equality from $i=j$
to $i=n$. We get:
\begin{align*}
\E[|y_j^n|^2]+ &\delta \sum_{i=j}^{n-1}\E[|z_i^n|^2]  +\kappa_n(1-\kappa_n) \sum_{i=j}^{n-1}\E[|u_i^n|^2]\\
\leq& \E[|\xi^n_n|^2]+2\delta \sum_{i=j}^{n-1}\E[y_i^n g(t_i,y_i^n,z_i^n,u_i^n)]+2
\sum_{i=j}^{n-1}\E[y_i^na_i^n]-2 \sum_{i=j}^{n-1}\E[y_i^nk_i^n], \\
\leq& \E[|\xi^n_n|^2]+\delta \sum_{i=j}^{n-1}g(t_i,0,0,0)^2
+\delta\left(1+2C_g+2C_g^2+\frac{2C_g^2 \delta}{\kappa_n(1-\kappa_n)}\right)
\sum_{i=j}^{n-1}\E[|y_i^n|^2] \\
&+\frac{\delta}{2}\sum_{i=j}^{n-1}\E[|z_i^n|^2]
+\frac{\kappa_n(1-\kappa_n)}{2}\sum_{i=j}^{n-1}\E[|u_i^n|^2]+ \frac{2\delta}{\alpha} \sum_{i=j}^{n-1} \E(|y^n_i|^2) + \frac{\alpha}{\delta} \sum_{i=j}^{n-1} \E(|a^n_i|^2)+ \frac{\alpha}{\delta} \sum_{i=j}^{n-1} \E(|k^n_i|^2).\\
\end{align*}

Since $\xi^n_i \le y^n_i \le \zeta^n_i$, we get
\begin{align}\label{estim_ak}
 & a^n_i \le \left( \E(\xi^n_{i+1}|\mG^n_i) +\delta
    g(t_i,\xi^n_i,z^n_i,u^n_i)-\xi^n_i\right)^-=
    \delta(b^{\xi}_{t_i}+g(t_i,\xi^n_i,z^n_i,u^n_i))^-,\\
    &k^n_i \le \left( \E(\zeta^n_{i+1}|\mG^n_i) +\delta
    g(t_i,\zeta^n_i,z^n_i,u^n_i)-\zeta^n_i\right)^+=
    \delta(b^{\zeta}_{t_i}+g(t_i,\zeta^n_i,z^n_i,u^n_i))^+.\notag
  \end{align}
  Then, using the Lipschitz property of $g$ gives
  \begin{align}\label{eq37}
    \frac{\alpha}{\delta} \sum_{i=j}^{n-1} \E(|a^n_i|^2)\le 5 \alpha \delta
    \sum_{i=j}^{n-1}\E[|b^{\xi}_i|^2+|g(t_i,0,0,0)|^2 + C_g^2( |\xi^n_i|^2 +
    |z^n_i|^2 + |u^n_i|^2)],
  \end{align}
  and the same result holds for $\frac{\alpha}{\delta} \sum_{i=j}^{n-1}
  \E(|k^n_i|^2)$.
  By Using Assumption \ref{hypo2} and the inequality $\sup_i \E(|y^n_i|^2) \le
  c$, we get
  \begin{align*} \delta \sum_{i=j}^{n-1}\E[|z_i^n|^2]  +\kappa_n(1-\kappa_n)
    \sum_{i=j}^{n-1}\E[|u_i^n|^2]\le c& + \delta \left( \frac{1}{2}+ 10 \alpha
      C_g^2\right)\sum_{i=j}^{n-1} \E(|z^n_i|^2)\\
    &+ \kappa_n(1-\kappa_n)\left( \frac{1}{2}+ 10 \alpha C_g^2
      \frac{\delta}{\kappa_n(1-\kappa_n)}\right)\sum_{i=j}^{n-1} \E(|u^n_i|^2).
  \end{align*}
  Since $\frac{\delta}{(1-\kappa_n)\kappa_n}=\frac{1}{\lambda}
  \frac{\lambda \delta}{(1-e^{-\lambda \delta})e^{-\lambda \delta}}$ and $e^x
  \le \frac{x e^{2x}}{e^x-1}\le e^{2x}$, we get
  $\frac{\delta}{(1-\kappa_n)\kappa_n}\le \frac{1}{\lambda}e^{2\lambda
    T}$. Then, by
  taking $\alpha=\frac{1}{40 C_g^2} (\lambda e^{-2\lambda T}\wedge 1)$, we get $\delta \sum_{i=j}^{n-1}\E[|z_i^n|^2]  +\kappa_n(1-\kappa_n)
    \sum_{i=j}^{n-1}\E[|u_i^n|^2]\le c$. Plugging this result in \eqref{eq37}
    ends the proof.

\end{proof}
The same type of proof gives the following Lemma
\begin{lemma}\label{lem8}
Under Assumption \ref{hypo2}, we have
\begin{align*}
\sup_j\E[ |\overline{y}_j^n|^2]+\E \left[\delta
  \sum_{j=0}^{n-1}|\overline{z}_j^n|^2 +\kappa_n(1-\kappa_n)
  \sum_{j=0}^{n-1}|\overline{u}_j^n|^2  + \frac{1}{\delta} \sum_{j=0}^{n-1}
  |\overline{a}_j^n |^2+ \frac{1}{\delta} \sum_{j=0}^{n-1}|\overline{k}_j^n |^2\right] \leq c.
\end{align*}
\end{lemma}

\subsection{Proof of the convergence of $\ov{\Theta}^n-\Theta^n$ and
  $\Theta^n-\Theta^{p,n}$}\label{sect:proof_prop}

\begin{proposition}\label{prop6}
Assume that Assumption \ref{hypo2} holds and $g$ is a Lipschitz driver. We have
\begin{align}\label{prop6_1}
\lim_{n \rightarrow \infty} \sup_{0 \leq t \leq T}\E[
|\overline{Y}_t^{n}-Y_t^{n}|^2]+\E[\int_0^T|\overline{Z}_s^{n}-Z_s^{n}|^2ds]+\E[\int_0^T|\overline{U}_s^{n}-U_s^{n}|^2 ds] = 0.
\end{align}
Moreover,
$\lim_{n\rightarrow \infty} (\overline{\alpha}^{n}_t-\alpha^{n}_t)=0$ in
$L^2(\mF_t)$, for $t\in [0,T]$.
\end{proposition}

\begin{proof}
Let us consider $y_j^n$, the solution of the discrete implicit reflected sheme \eqref{scheme2} and $\overline{y}_j^n$, the solution of the explicit reflected scheme \eqref{scheme2_exp}.
We compute $|y_j^n-\overline{y}_j^n|^2$, we take the expectation and we get:
\begin{align*}
\E[|y_j^n-\overline{y}_j^n|^2] \leq& \E[|y_{j+1}^n-\overline{y}_{j+1}^n|^2]-\delta \E[|z_j^n-\overline{z}_j^n|^2]-\kappa_n(1-\kappa_n) \E[|u_j^n-\overline{u}_j^n|^2]\\
&+2\delta \E[(y_j^n-\overline{y}_j^n)(g(t_j,y_j^n,z_j^n,u_j^n)-g(t_j,\E[\overline{y}_{j+1}^n|\mathcal{F}_j^n],\overline{z}_j^n,\overline{u}_j^n))]\\
&-\E\left[\delta(g(t_j,y_j^n,z_j^n,u_j^n)-g(t_j,\E[\overline{y}_{j+1}^n|\mathcal{F}_j^n],\overline{z}_j^n,\overline{u}_j^n))+(a_j^n-\overline{a}_j^n)-(k_j^n-\overline{k}_j^n)\right]^2\\
&+2\E[(y_j^n-\overline{y}_j^n)(a_j^n-\overline{a}_j^n)]-2\E[(y_j^n-\overline{y}_j^n)(k_j^n-\overline{k}_j^n)],\\
&\leq \E[|y_{j+1}^n-\overline{y}_{j+1}^n|^2]-\delta
\E[|z_j^n-\overline{z}_j^n|^2]-\kappa_n(1-\kappa_n)
\E[|u_j^n-\overline{u}_j^n|^2]\\
&+2\delta \E[(y_j^n-\overline{y}_j^n)(g(t_j,y_j^n,z_j^n,u_j^n)-g(t_j,\E[\overline{y}_{j+1}^n|\mathcal{F}_j^n],\overline{z}_j^n,\overline{u}_j^n))].
\end{align*}

The last inequality comes from $ (y_j^n-\overline{y}_j^n)(a_j^n-\overline{a}_j^n) \leq 0$ and $
(y_j^n-\overline{y}_j^n)(k_j^n-\overline{k}_j^n) \geq 0$ (this ensues from the
third and fourth lines of 
$(\mathcal{S}_1)$ and $(\ov{\mathcal{S}}_1)$). Taking the sum from
$j=i$ to $n-1$ we get
\begin{align}
  \E[|y_i^n-\overline{y}_i^n|^2]+&\delta \sum_{j=i}^{n-1}
  \E[|z_j^n-\overline{z}_j^n|^2]+ \kappa_n(1-\kappa_n) \sum_{j=i}^{n-1}
  \E[|u_j^n-\overline{u}_j^n|^2]\notag\\ &\le 2\delta \sum_{j=i}^{n-1}
  \E[(y_j^n-\overline{y}_j^n)(g(t_j,y_j^n,z_j^n,u_j^n)-g(t_j,\E[\overline{y}_{j+1}^n|\mathcal{F}_j^n],\overline{z}_j^n,\overline{u}_j^n))],\notag
  \\
  &\le 2\delta C_g \sum_{j=i}^{n-1} \E\left[|y_j^n-\overline{y}_j^n|
  |y_j^n-\E[\overline{y}_{j+1}^n|\mathcal{F}_j^n]|\right]+2\delta
  C_g^2\left(1+\frac{\delta}{\kappa_n(1-\kappa_n)}\right) \sum_{j=i}^{n-1}
  \E[|y_j^n-\overline{y}_j^n|^2]\notag \\
  &+\frac{\delta}{2}\sum_{j=i}^{n-1}\E[|z_j^n-\overline{z}_j^n|^2]+\frac{\kappa_n(1-\kappa_n)}{2}\sum_{j=i}^{n-1}\E[|u_j^n-\overline{u}_j^n|^2]\label{eq35}.
\end{align}
Since $y^n_j-\E[\overline{y}_{j+1}^n|\mathcal{F}_j^n]=y^n_j-\ov{y}^n_j+\ov{y}_j^n-\E[\overline{y}_{j+1}^n|\mathcal{F}_j^n]=y^n_j-\ov{y}^n_j+\delta
g(t_j,\E[\overline{y}_{j+1}^n|\mathcal{F}_j^n],\overline{z}_j^n,\overline{u}_j^n)+\ov{a}^n_j-\ov{k}^n_j$,
we get
\begin{align*}
 2 \delta C_g  \E\left[|y_j^n-\overline{y}_j^n|
    |y_j^n-\E[\overline{y}_{j+1}^n|\mathcal{F}_j^n]|\right] \le (2
  C_g+1)\delta \E[|y_j^n-\overline{y}_j^n|^2]+C_g^2 \delta \E\left[\left(|\delta g(t_j,\E[\overline{y}_{j+1}^n|\mathcal{F}_j^n],\overline{z}_j^n,\overline{u}_j^n)|+|\ov{a}^n_j|+|\ov{k}^n_j|\right)^2\right].
\end{align*} 
Plugging the previous inequality in \eqref{eq35} and using Lemma \ref{lem8}
gives
\begin{align*} \E[|y_i^n-\overline{y}_i^n|^2]+\frac{\delta}{2}
  \sum_{j=i}^{n-1} \E[|z_j^n-\overline{z}_j^n|^2]+&
  \frac{\kappa_n(1-\kappa_n)}{2} \sum_{j=i}^{n-1}
  \E[|u_j^n-\overline{u}_j^n|^2]\\ &\le \left(1+2C_g+2C_g^2+\frac{2C_g^2
      \delta}{\kappa_n(1-\kappa_n)}\right)\delta
  \sum_{j=i}^{n-1}\E[|y_j^n-\overline{y}_j^n|^2]+ c \delta^2.
\end{align*}
Let $n$ be bigger than $N_0$, then $\delta\left(1+2C_g+2C_g^2+\frac{2 \delta C_g^2 \delta}{\kappa_n(1-\kappa_n)}\right)<1$ (for all $n\ge
1$ we have
$\frac{\delta}{\kappa_n(1-\kappa_n)} \le \frac{1}{\lambda} e^{2\lambda T}$).

The assumption on $\delta$ enables to apply Gronwall's Lemma to
get $\sup_{0\le i \le n} \E[|y_i^n-\overline{y}_i^n|^2] \le c
\delta^2$. Plugging this result in the previous inequality leads to
\eqref{prop6_1}. The convergence of $(A^n-K^n)- (\ov{A}^n-\ov{K}^n)$ ensues
from
\begin{align*}
  A^n_t-K^n_t=Y^n_0-Y^n_t-\int_0^t g(s,Y^n_s,Z^n_s,U^n_s)ds +\int_0^t Z^n_s
  dW^n_s+\int_0^t U^n_s d\tilde{N}^n_s,\\
  \ov{A}^n_t-\ov{K}^n_t=\ov{Y}^n_0-\ov{Y}^n_t-\int_0^t g(s,\ov{Y}^n_s,\ov{Z}^n_s,\ov{U}^n_s)ds +\int_0^t \ov{Z}^n_s
  dW^n_s+\int_0^t \ov{U}^n_s d\tilde{N}^n_s,\\
\end{align*}
from the Lipschitz property of $g$ and from \eqref{prop6_1}.
\end{proof}

\begin{proposition}\label{prop7}
  Assume that Assumption \ref{hypo2} holds and $g$
  is a Lipschitz driver. For $n \ge N_0$, we get
  \begin{align}\label{eq19} \sup_{0 \leq t \leq
      T}\E[
    |{Y}_t^{n}-Y_t^{p,n}|^2]+\E[\int_0^T|{Z}_s^{n}-Z_s^{p,n}|^2ds]+\E[\int_0^T|{U}_s^{n}-U_s^{p,n}|^2
    ds]\le \frac{c}{\sqrt{p}}.
\end{align}
Moreover, $\forall \; t \in [0,T]$,
$ \E[|{\alpha}^{n}_t-\alpha^{p,n}_t|^2] \le \frac{c}{\sqrt{p}}$.
\end{proposition}

\begin{proof}
Let us first prove \eqref{eq19}. From \eqref{scheme2}, \eqref{discrete} and
Lemma \ref{lem12} applied to the process $(y^n-y^{p,n})$ following the
beginning of the proof of Lemma \ref{Reflscheme}, we get
\begin{align*}
\E|y_j^n-y_j^{p,n}|^2+\delta\sum_{i=j}^{n-1} \E |z_i^n-z_i^{p,n}|^2&+(1-\kappa_n)\kappa_n\sum_{i=j}^{n-1} \E [|u_i^n-u_i^{p,n}|^2]+(1-\kappa_n)\kappa_n\sum_{i=j}^{n-1} \E [|v_i^n-v_i^{p,n}|^2]\\
&=2\sum_{i=j}^{n-1} \E[(y_i^{n}-y_i^{p,n})(g(t_i, y_i^n, z_i^n,u_{i}^n)-g(t_i, y_i^{p,n}, z_i^{p,n},u_{i}^{p,n})) \delta]\\
&+ 2\sum_{i=j}^{n-1} \E[(y_{i}^{n}-y_{i}^{p,n})(a_i^{n}-a_i^{p,n})]-2\sum_{i=j}^{n-1}\E[(y_i^n-y_i^{p,n})(k_i^n-k_i^{p,n})].
\end{align*}

Let us deal with the last two terms

\begin{equation*}
(y_{i}^n-y_{i}^{p,n})(a_i^n-a_{i}^{p,n})=(y_{i}^{n}-\xi_{i}^{n})a_{i}^{n}-(y_{i}^{p,n}-\xi_{i}^n)a_{i}^{n}-(y_{i}^{n}-\xi_{i}^{n})a_{i}^{p,n}+(y_{i}^{p,n}-\xi_{i}^n)a_{i}^{p,n} \leq (y_{i}^{p,n}-\xi_{i}^{n})^{-}a_{i}^{n}.
\end{equation*}
By using same computations, we derive
\begin{equation*}
(y_{i}^n-y_{i}^{p,n})(k_i^n-k_{i}^{p,n})\geq -(y_{i}^{p,n}-\zeta_{i}^{n})^+ k_{i}^{n}.
\end{equation*}

By using the Lipschitz property of $g$, we get

\begin{align*}
\E[|y_j^n-y_j^{p,n}|^2]&+\frac{1}{2}\delta
\E[|z_j^n-z_j^{p,n}|^2]+\frac{\kappa_n(1-\kappa_n)}{2}\E[|u_j^n-u_j^{p,n}|^2]\\
\leq& \left(2C_g+2C_g^2+\frac{2C_g^2
  \delta}{\kappa_n(1-\kappa_n)}\right)\delta\sum_{i=j}^{n-1}
\E[(y_i^n-y_i^{p,n})^2]+2\sum_{i=j}^{n-1}\E[(y_i^{p,n}-\xi_i^n)^{-}a_i^n+(y_i^{p,n}-\zeta_i^n)^+k_i^n].\\
\end{align*}

  Using Cauchy-Schwarz inequality gives

  \begin{align*}
   &\E[|y_j^n-y_j^{p,n}|^2]+\frac{1}{2}\delta
    \E[|z_j^n-z_j^{p,n}|^2]+\frac{\kappa_n(1-\kappa_n)}{2}\E[|u_j^n-u_j^{p,n}|^2]\\
& \leq \left(2C_g+2C_g^2+\frac{2C_g^2 \delta}{\kappa_n(1-\kappa_n)}\right)\delta\sum_{i=j}^{n-1}\E[(y_i^n-y_i^{p,n})^2]\\
&+ 2 \left(\delta \sum_{i=j}^{n-1} \E \left[\left((y_i^{p,n}-\xi_i^n)^-\right)^2\right]
\right)^{\frac{1}{2}}\left(\frac{1}{\delta} \sum_{i=j}^{n-1}\E
  [(a^n_i)^2]  \right)^{\frac{1}{2}}+ 2 \left(\delta \sum_{i=j}^{n-1}\E \left[\left((y_i^{p,n}-\zeta_i^n)^+\right)^2\right] \right)^{\frac{1}{2}}\left(\frac{1}{\delta}\sum_{i=j}^{n-1}\E[(k^n_i)^2]  \right)^{\frac{1}{2}},\\
\leq &\left(2C_g+2C_g^2+\frac{2C_g^2
    \delta}{\kappa_n(1-\kappa_n)}\right)\delta
\sum_{i=j}^{n-1}E[(y_i^n-y_i^{p,n})^2]\\
&+ \frac{2}{\sqrt p} \left(\frac{1}{p \delta} \sum_{i=j}^{n-1}\E[(a_i^{p,n})^2]  \right)^{\frac{1}{2}}\left(\frac{1}{\delta} \sum_{i=j}^{n-1}\E
  [(a^n_i)^2]  \right)^{\frac{1}{2}}+\frac{2}{\sqrt p} \left(\frac{1}{p \delta} \sum_{i=j}^{n-1}\E[(k_i^{p,n})^2]  \right)^{\frac{1}{2}}\left(\frac{1}{\delta}\sum_{i=j}^{n-1}\E[(k^n_i)^2]  \right)^{\frac{1}{2}}.
\end{align*}


Since $n \ge N_0$, Lemma \ref{Reflscheme}, Lemma \ref{estimationpenal} and Gronwall
inequality give \eqref{eq19}. Concerning $\alpha^n_t-\alpha^{p,n}_t$ we have
\begin{align*}
  \alpha^n_t-\alpha^{p,n}_t=&(Y^n_t-Y^{p,n}_t)-(Y^n_0-Y^{p,n}_0)-\int_0^t
  g(s,Y^n_s,Z^n_s,U^n_s)-g(s,Y^{p,n}_s,Z^{p,n}_s,U^{p,n}_s)ds \\
  &+\int_0^t
  (Z^n_s-Z^{p,n}_s) dW^n_s+\int_
  0^t (U^n_s-U^{p,n}_s)d \tilde{N}^n_s.
\end{align*}
It remains to take the square of both sides, then the expectation, and to use
the Lipschitz property of $g$ combining with \eqref{eq19} to get the result.

\end{proof}

\section{Numerical simulations}

We consider the simulation of the solution of a DRBSDE with obstacles and
driver of the following form:
$\xi_t:={(W_t)}^2+2(1-\frac{t}{T})\Tilde{N}_t+\frac{1}{2}(T-t),
\zeta_t:={(W_t)}^2+(1-\frac{t}{T})((\Tilde{N}_t)^2+1)+\frac{1}{2}(T-t),
g(t,\omega,y,z,u):=-5|y+z|+6u$.\\

Table \ref{tab1} gives the values of $Y_0$ with respect to $n$. We notice that
the algorithm converges quite fast in $n$. Moreover, the computational
time is low.

\begin{table}[htbp]
\caption{\small The solution ${y}^{n}$ at time $t=0$}\label{tab1}
\centerline{
\begin{tabular}{|c||c|c|c|c|c|c|c|}
  \hline
n & 10 & 20 & 50 & 100 & 200 & 300 & 400 \\
\hline
\hline
$y_0^{n}$ & 1.2191 & 1.3238 & 1.3953 & 1.4167 & 1.4293 & 1.4332 & 1.4352 \\
\hline
CPU time & $2.14 \times 10^{-4}$ & $1.5 \times 10^{-3}$ & 0.0211 & 0.1622 &
1.4230 & 5.2770 & 12.5635\\
 \hline
\end{tabular}}
\end{table}
When we use the explicit penalized scheme introduced in \cite{DL14}, we get
$y_0^{p,n}=1.4353$ for $n=400$ and $p=20000$. The CPU time is $12.85$s.

Figures \ref{img1}, \ref{img2} and \ref{img3} represent one path the Brownian
motion, one path of the compensated Poisson process (with $\lambda=5$) and the
corresponding path of $({y}^{n}_i,{\xi}^n_i,\zeta^n_i)_{1 \le i \le n}$. We
notice that for all $i$, ${y}^{n}_i$ stays between the two obstacles. The
values of $y_0^n$ and $y_0^{p,n}$ are almost the same when $n=400$ and
$p=20000$. The CPU times are also of the same order. The main advantage of the
reflected scheme is that there is only one parameter to tune ($n$).

\begin{figure}[htbp]
\centerline{
\includegraphics[width=0.75\textwidth,height=6cm]{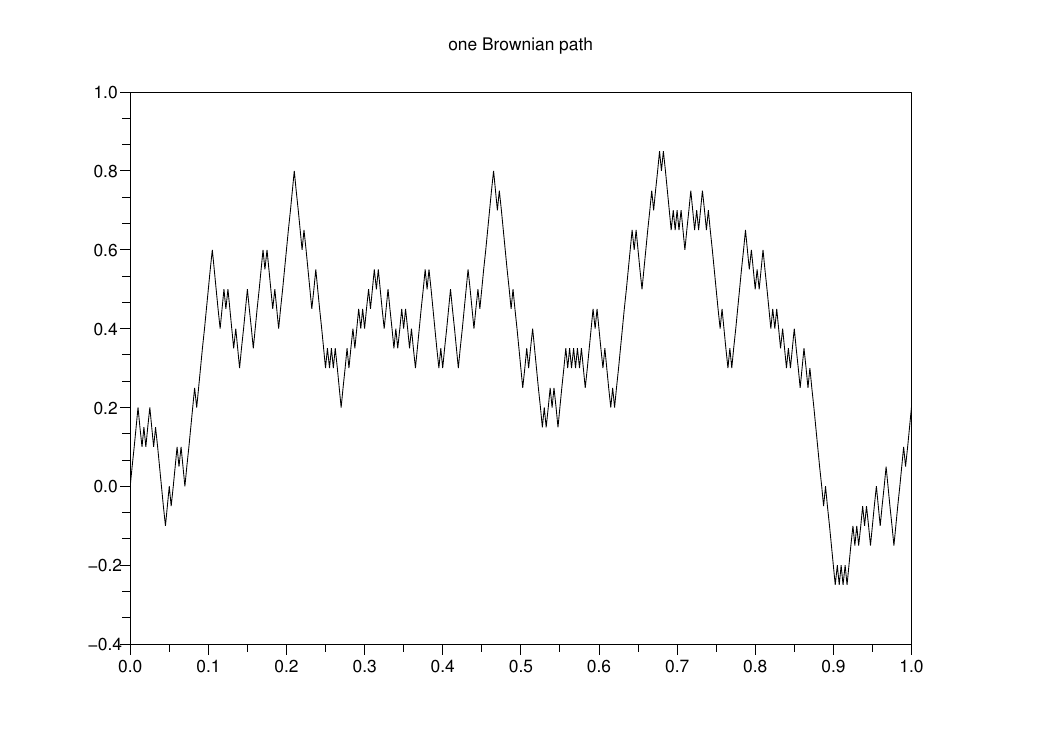}}
\caption{\small One path of the Brownian motion for $n=400$.}
\label{img1}
\end{figure}

\begin{figure}[htbp]
\centerline{
\includegraphics[width=0.75\textwidth,height=6cm]{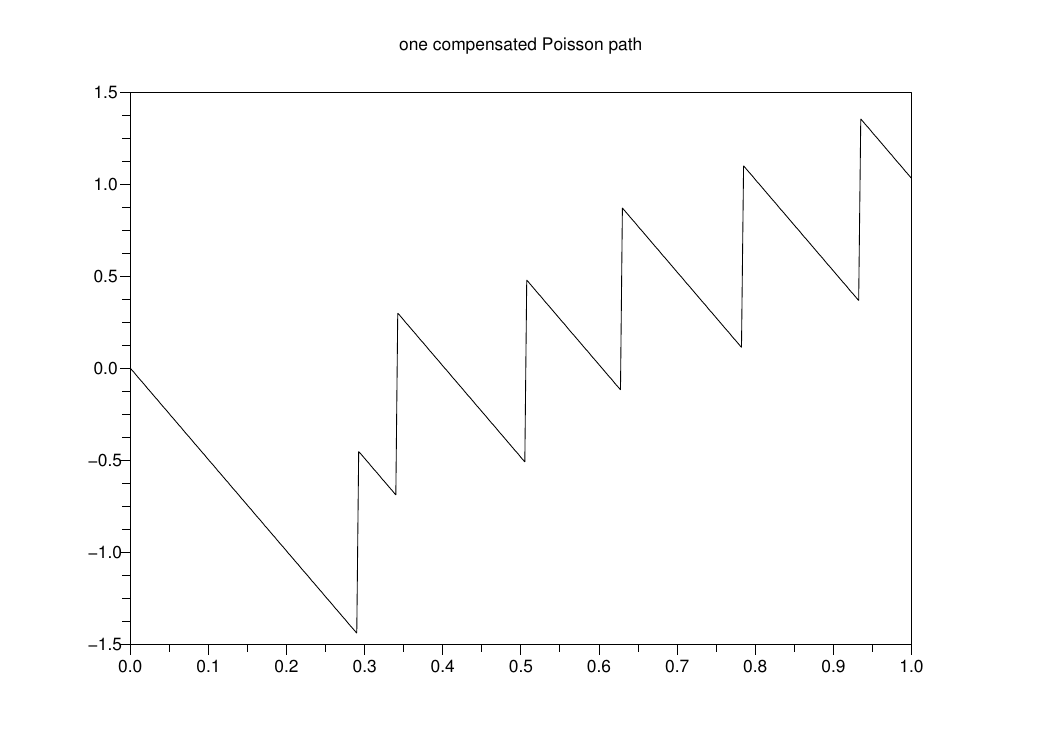}}
\caption{\small One path of the compensated Poisson process for $\lambda=5$ and $n=400$.}
\label{img2}
\end{figure}

\begin{figure}[htbp]
\centerline{
\includegraphics[width=0.75\textwidth,height=6cm]{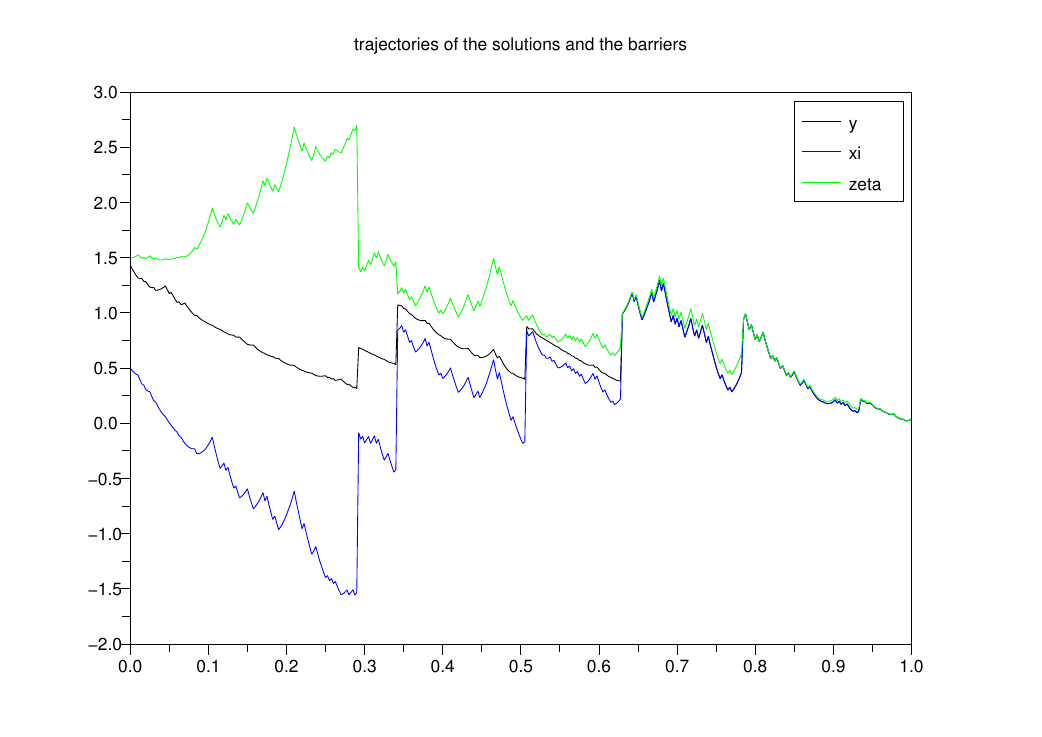}}
\caption{\small Trajectories of the solution ${y}^{n}$ and the
  barriers ${\xi}^n$ and ${\zeta}^n$ for $\lambda=5$ and $n=400$.}
\label{img3}
\end{figure}

\appendix

\section{Technical result for the implicit penalized scheme}

In this Section, we use $N_0$ and $c$ introduced in Definition \ref{def1}.

\begin{lemma}\label{estimationpenal} Suppose Assumption \ref{hypo2} holds and
  $g$ is a Lipschitz driver.
For each $p \in \mathbb{N}$ and $n \ge N_0$ we have
\begin{align*}
\sup_j\E[|y_j^{p,n}|^2]+\delta \sum_{j=0}^{n-1}\E[|z_j^{p,n}|^2]+\kappa_n(1-\kappa_n)\sum_{j=0}^{n-1}\E[|u_j^{p,n}|^2]+\frac{1}{p \delta} \sum_{j=0}^{n-1}\E[|a_j^{p,n}|^2]+\frac{1}{p \delta} \sum_{j=0}^{n-1}\E[|k_j^{p,n}|^2] \leq c.
\end{align*}
\end{lemma}
\begin{proof}
By applying Lemma \ref{lem12} to the process $y^{p,n}$ between $i$ and $i+1$
and by suming the equality from $i=i$ to $i=n$, we get
\begin{align*}
\E[|y_j^{p,n}|^2]+\delta \sum_{i=j}^{n-1}\E[|z_i^{p,n}|^2]&+\kappa_n(1-\kappa_n)\sum_{i=j}^{n-1}\E[|u_i^{p,n}|^2]+\kappa_n(1-\kappa_n)\sum_{i=j}^{n-1}\E[|v_i^{p,n}|^2]\\
&\leq \E[|\xi^n_n|^2]+2\sum_{i=j}^{n-1} \E[|y_i^{p,n}||g(t_i,y_i^{p,n},z_i^{p,n},u_i^{p,n})\delta|]+2\E[\sum_{i=j}^{n-1}(y_i^{p,n}a_i^{p,n}-y_i^{p,n}k_i^{p,n})].
\end{align*}
Note that $y_i^{p,n}a_i^{p,n}=-\frac{1}{p\delta}(a_i^{p,n})^2+\xi_i^na_i^{p,n}$ and $y_i^{p,n}k_i^{p,n}=\frac{1}{p\delta}(k_i^{p,n})^2+\zeta_i^nk_i^{p,n}$.
We have that:
\begin{align*}
\E[|y_j^{p,n}|^2]&+\frac{\delta}{2} \sum_{i=j}^{n-1}\E[|z_i^{p,n}|^2]+\frac{\kappa_n(1-\kappa_n)}{2}\sum_{i=j}^{n-1}\E[|u_i^{p,n}|^2]+\frac{1}{p
  \delta} \sum_{i=j}^{n-1}\E[|a_i^{p,n}|^2]+\frac{1}{p \delta}
\sum_{i=j}^{n-1}\E[|k_i^{p,n}|^2]\\
&\leq \E[|\xi^n_n|^2]+\delta \E[\sum_{i=j}^{n-1}|g(t_i,0,0,0)|^2]+2\delta
\left(1+2C_g+2C_g^2+\frac{2C_g^2
    \delta}{\kappa_n(1-\kappa_n)}\right)\sum_{i=j}^{n-1}\E[|y_i^{p,n}|^2]\\
&+2\sum_{i=j}^{n-1}\E[(\xi_i^n)a_i^{p,n}]-2\sum_{i=j}^{n-1}\E[(\zeta_i^n)k_i^{p,n}].
\end{align*}
We get
$2\sum_{i=j}^{n-1}\E[(\xi_i^n)a_i^{p,n}]\le \alpha \E(\sup_i |\xi^n_i|^2)+
\frac{1}{\alpha}\E\left(\sum_{i=j}^{n-1} a_i^{p,n}\right)^2$ and $2\sum_{i=j}^{n-1}\E[(\zeta_i^n)k_i^{p,n}]\le \beta \E(\sup_i |\zeta^n_i|^2)+
\frac{1}{\beta}\E\left(\sum_{i=j}^{n-1} k_i^{p,n}\right)^2$. Following the
same type of proof as \cite[Lemma 2]{LSM04}, we get
\begin{align*}
  \E\left(\sum_{i=j}^{n-1} a_i^{p,n}\right)^2+\E\left(\sum_{i=j}^{n-1}
    k_i^{p,n}\right)^2 \le C(c+\E[\sum_{i=j}^{n-1} \delta(|y^{p,n}_i|^2 +
  |z^{p,n}_i|^2)+\kappa_n(1-\kappa_n) (|u^{p,n}_i|^2+|v^{p,n}_i|^2)].
\end{align*}

Finally, by taking $\alpha=\beta=4C$ and by applying the Gronwall inequality
(we recall $n \ge N_0$), we get that:
\begin{align*}
\sup_j \E[|y_j^{p,n}|^2+\frac{\delta}{4}\sum_{j=0}^{n-1}|z_j^{p,n}|^2+\frac{\kappa_n(1-\kappa_n)}{4}\sum_{j=0}^{n-1}|u_j^{p,n}|^2+\frac{1}{p \delta} \sum_{j=0}^{n-1}|a_j^{p,n}|^2+\frac{1}{p \delta} \sum_{j=0}^{n-1}|k_j^{p,n}|^2] \leq c.
\end{align*}

\end{proof}

\section{Some results on discrete stochastic calculus}
In this section we present two lemmas which are used throughout the paper. 

\begin{lemma}\label{lem12} Consider two integers $i_0$ and $i_1$ in ${0,...,N}$
  and $(y_n)_n$ a discrete process. We have
\begin{align*}
y_{i_1}^2= y_{i_0}^2+2y_{i_0}(y_{i_1}-y_{i_0})+(y_{i_1}-y_{i_0})^2.
\end{align*}

\end{lemma}
The proof comes from the computation of $((b-a)+a)^2$, we omit it.

\begin{lemma}(A discrete Gronwall lemma)
Let $a$, $b$ and $\alpha$ be positive constants, $\delta b<1$ and a sequence $(v_j)_{j=1,...n}$ of positive numbers  such that for every $j$
$$v_j+\alpha \leq a+b \delta \sum_{i=1}^{j}v_i. $$
Then
$$ \sup_{j \leq n}v_j +\alpha \leq a e^{bT}.$$
\end{lemma}
A proof of this lemma can be found in \cite[Lemma 2.2]{MPX02}, so we omit it.



\bibliographystyle{abbrv}
\bibliography{ref}

\begin{thebibliography}{10}

\bibitem{BE08}
B.~Bouchard and R.~Elie.
\newblock Discrete-time approximation of decoupled {F}orward-{B}ackward {SDE}
  with jumps.
\newblock {\em Stochastic Processes and their Applications}, (118):53--75,
  2008.

\bibitem{C09}
J.-F. Chassagneux.
\newblock A discrete-time approximation for doubly reflected {BSDE}s.
\newblock {\em Adv. in Appl. Probab.}, 41(1):101--130, 2009.

\bibitem{CM08}
S.~Cr{\'e}pey and A.~Matoussi.
\newblock Reflected and doubly reflected {BSDE}s with jumps: a priori estimates
  and comparison.
\newblock {\em Ann. Appl. Probab.}, 18(5):2041--2069, 2008.

\bibitem{CK96}
J.~Cvitanic and I.~Karatzas.
\newblock Backward stochastic differential equations with reflection and
  {D}ynkin games.
\newblock {\em The Annals of Probability}, (41):2024--2056, 1996.

\bibitem{DL14}
R.~Dumitrescu and C.~Labart.
\newblock Numerical approximation of doubly reflected {BSDE}s with jumps and
  {RCLL} obstacles.
\newblock \url{http://hal.archives-ouvertes.fr/hal-01006131}, 2014.

\bibitem{DQS14}
R.~Dumitrescu, M.~Quenez, and A.~Sulem.
\newblock Double barrier reflected {BSDE}s with jumps and generalized {D}ynkin
  games.
\newblock 2014.

\bibitem{EKPPQ97}
N.~El~Karoui, C.~Kapoudjian, E.~Pardoux, S.~Peng, and M.~Quenez.
\newblock Reflected solutions of {B}ackward {SDE}'s and related obstacle
  problems for {PDE}'s.
\newblock {\em The Annals of Probability}, 25(2):702--737, 1997.

\bibitem{Ess08}
E.~Essaky.
\newblock Reflected backward stochastic differential equation with jumps and
  {RCLL} obstacle.
\newblock {\em Bulletin des Sciences Math{\'e}matiques}, (132):690--710, 2008.

\bibitem{EHO05}
E.~Essaky, N.~Harraj, and Y.~Ouknine.
\newblock Backward stochastic differential equation with two reflecting
  barriers and jumps.
\newblock {\em Stochastic Analysis and Applications}, 23:921--938, 2005.

\bibitem{HH06}
S.~Hamad\`ene and M.~Hassani.
\newblock {BSDE}s with two reacting barriers driven by a {B}rownian motion and
  an independent {P}oisson noise and related {D}ynkin game.
\newblock {\em Electronic Journal of Probability}, 11:121--145, 2006.

\bibitem{HO03}
S.~Hamad\`ene and Y.~Ouknine.
\newblock Reflected {B}ackward {S}tochastic {D}ifferential {E}quation with
  jumps and random obstacle.
\newblock {\em Elec. Journ. of Prob.}, 8:1--20, 2003.

\bibitem{HO13}
S.~Hamadène and Y.~Ouknine.
\newblock {Reflected Backward SDEs with general jumps}.
\newblock \url{http://arxiv.org/abs/0812.3965}, 2013.

\bibitem{HW09}
S.~Hamadène and H.~Wang.
\newblock {BSDE}s with two {RCLL} {R}eflecting {O}bstacles driven by a
  {B}rownian {M}otion and {P}oisson {M}easure and related {M}ixed {Z}ero-{S}um
  {G}ames.
\newblock {\em Stochastic Processes and their Applications}, 119:2881--2912,
  2009.

\bibitem{JS_03}
J.~Jacod and A.~N. Shiryaev.
\newblock {\em Limit theorems for stochastic processes}, volume 288 of {\em
  Grundlehren der Mathematischen Wissenschaften [Fundamental Principles of
  Mathematical Sciences]}.
\newblock Springer-Verlag, Berlin, second edition, 2003.

\bibitem{LMT07}
A.~Lejay, E.~Mordecki, and S.~Torres.
\newblock {Numerical approximation of Backward Stochastic Differential
  Equations with Jumps}.
\newblock \url{https://hal.inria.fr/inria-00357992}, Sept. 2014.

\bibitem{LSM04}
J.~Lepeltier and J.~San~Martin.
\newblock Backward {SDE}s with two barriers ans continuous coefficient: an
  existence result.
\newblock {\em Journal of Appl. Prob.}, (41):162--175, 2004.

\bibitem{MPX02}
J.~Mémin, S.~Peng, and M.~Xu.
\newblock Convergence of solutions of discrete {R}eflected backward {SDE}'s and
  {S}imulations.
\newblock {\em Acta Mathematica Sinica}, 24(1):1--18, 2002.

\bibitem{PP90}
E.~Pardoux and S.~Peng.
\newblock Adapted solution of a backward stochastic differential equation.
\newblock {\em Systems Control Lett.}, 14(1):55--61, 1990.

\bibitem{PX11}
S.~Peng and M.~Xu.
\newblock Numerical algorithms for {BSDE}s with 1-d {B}rownian motion:
  convergence and simulation.
\newblock {\em ESAIM: Mathematical Modelling and Numerical Analysis},
  (45):335--360, 2011.

\bibitem{QS13}
M.~Quenez and A.~Sulem.
\newblock {BSDE}s with jumps, optimization and applications to dynamic risk
  measures.
\newblock {\em Stochastic Processes and their Applications}, (123):3328--3357,
  2013.

\bibitem{TL94}
S.~Tang and X.~Li.
\newblock Necessary conditions for optimal control of stochastic systems with
  random jumps.
\newblock {\em SIAM J. Cont. and Optim.}, 32:1447--1475, 1994.

\bibitem{Xu11}
M.~Xu.
\newblock Numerical algorithms and {S}imulations for {R}eflected {B}ackward
  {S}tochastic {D}ifferential {E}quations with {T}wo {C}ontinuous {B}arriers.
\newblock {\em Journal of Computational and Applied Mathematics},
  236:1137--1154, 2011.

\end{thebibliography}

\end{document}